\pdfoutput=1
\documentclass[12pt]{scrartcl}

\usepackage[scaled=.90]{helvet}
\usepackage{mathptmx} 
\usepackage{courier}
\usepackage[ngerman,english]{babel}
\usepackage{dsfont}
\usepackage{amsmath}
\usepackage{amssymb}

\usepackage{enumerate, amssymb, amsmath, amsthm}
\newcommand{\N}{\ensuremath{{\mathbb N}}}

\usepackage{tikz} 
\usepackage[T1]{fontenc}
\usepackage{lmodern}
\usepackage[figurewithin=section]{caption}

\newtheorem{satz}{Theorem}[section]

\newtheorem{lem}[satz]{Lemma} 
\newtheorem{prop}[satz]{Proposition} 
 
\newtheorem{koro}[satz]{Corollary} 
 
\newtheorem{anmerk}[satz]{Remark}

 \newcommand{\ov}[1]{{\overline #1}}
\newcommand{\R}{\ensuremath{{\mathbb R}}}
\newcommand{\E}{\ensuremath{{\mathbb E}}}
\newcommand{\Pro}{\ensuremath{{\mathbb P}}}
\newcommand{\Ha}{\ensuremath{{\mathcal H}}}
\newcommand{\X}{\ensuremath{{\mathcal X}}}
\newcommand{\A}{\ensuremath{{\mathcal A}}}

\begin{document}

\title{On the Solution of General Impulse Control Problems Using Superharmonic Functions}
\author{S\"oren Christensen\thanks{Christian-Albrechts-Universit\"at, Mathematisches Seminar, Kiel, Germany, email: {christensen}@math.uni-kiel.de.}}
\date{\today}
\maketitle

\begin{abstract}
In this paper, a characterization of the solution of impulse control problems in terms of superharmonic functions is given. In a general Markovian framework, the value function of the impulse control problem is shown to be the minimal function in a convex set of superharmonic functions. This characterization also leads to optimal impulse control strategies and can be seen as the corresponding characterization to the description of the value function for optimal stopping problems as a smallest superharmonic majorant of the reward function. The results are illustrated with examples from different fields, including multiple stopping and optimal switching problems.
\end{abstract}

\textbf{Keywords:} impulse control strategies; superharmonic functions, general Markov processes \vspace{.8cm}

\textbf{Subject Classifications:} 49N25, 60G40.

\section{Introduction}
Stochastic control techniques play a major role in many fields of applied probability. In particular, the developments in mathematical finance have stimulated the activities in this branch of control theory in the last decades. Many of these approaches have the disadvantage that they lead to non-realizable optimal strategies since these strategies consist of interventions at each time instant in a continuous time model. The right mathematical framework to consider discrete interventions in a continuous time model is given by impulse control problems. \\
Impulse control problems have been studied for decades. It seems to be impossible to give an overview over all fields of application and all different variants that have been used. We only want to mention finance, e.g. cash management and portfolio optimization, see \cite{K} and \cite{PaStettner}, optimal forest management, see \cite{W}, \cite{A04} and the references therein, and control of an exchange rate by the Central Bank, see \cite{MO}, \cite{CZ}. Most of these articles are based on the seminal work developed in \cite{BL}, which still turns out to be the main reference for theoretical results in this field. For underlying diffusion process under some further assumptions, the value function is proved to be a solution of a corresponding quasi-variational inequality, that also characterizes the optimal strategy. A more recent overview over results for jump-diffusions is given in \cite{OS}, see also \cite{K} for a survey with focus on financial applications.\\
On the other hand, it it known that there is a strong connection between impulse control problems and problems of optimal stopping. Under certain conditions, the value function of the impulse control problem can be found as the limit of a sequence of value functions for associated optimal stopping problems, see \cite[Chapter 7]{OS}. Moreover, the value function of the impulse control problem can be characterized as a solution to an implicit problem of optimal stopping, where implicit means that the reward function in the optimal stopping problem contains this value function itself, see \cite{K}.\\
For Markovian problems of optimal stopping, the most flexible and valuable approach -- both from a theoretical and practical point of view -- seems to be the superharmonic characterization of the value function; more precisely, under minimal condition, the value function is the smallest superharmonic function majorizing the   reward function. This characterization goes back to Dynkin (\cite{Dy63}) and turned out to be the right formulation for most such problems. For an explicit solution, this approach can be translated into free-boundary problems, which can be solved in many problems of interest. An excellent overview over recent developments in this field is given in the monograph \cite{PS}. 
One of main advantages of considering superharmonic functions (instead of, e.g. using a formulation using quasi-variational inequalities) is that regularity conditions can often be stated in a more natural way from a stochastic point of view. \\
One of the consequences of this superharmonic characterization is that optimal stopping problems for an underlying one-dimensional diffusion process can be solved explicitly in many situations of interest, since the superharmonic functions turn out to be transformed concave functions, see \cite{DK}, \cite{BLe}, or \cite{CI} for recent treatments. Therefore, one can say that optimal stopping of one-dimensional diffusion processes is well-understood. Inspired by these result, in the last years, different authors considered special classes of impulse control problems for an underlying one-dimensional diffusion processes, and obtained a solution in terms of superharmonic (resp. excessive) functions, see \cite{A}, \cite{AL}, and \cite{E}. One of the main advantages of these approaches is that they work in a very general setting without strong regularity assumptions on the problem, that are often needed for applying alternative approaches. \\

The question arises whether there is also a general characterization of the value function of an impulse control problem as the smallest function in a set of superharmonic functions, as for optimal stopping problems. The aim of this article is to consider the impulse control problem from a purely superharmonic point of view to use the well-known advantages for optimal stopping problems also for impulse control problems. This is carried out in a very general Markovian setting in the following section. The main results are Theorems \ref{thm:measurable} and \ref{thm:minimizer}, that give a characterization of the value function as well as the existence and description of an optimal strategy under very general conditions: Under natural assumptions, (for the problem without integral term) the value function of an impulse control can be characterized as the smallest $r$-superharmonic function $h$ with $Mh\leq h$, where $M$ denotes the maximum operator. This can be seen as a consequent dual approach to impulse control: The maximization over impulse control strategies is transformed into a minimization problem over superharmonic functions. It turns out that it is possible to follow a similar line of argument as for optimal stopping problems, although some refinements of the arguments are needed of course. To see the connection, we use a similar presentation as given in  \cite[Chapter 1]{PS}. To the best of the author's knowledge, the presented approach is new in the literature, although many connected results are already known. Some of these connections are discussed in Section \ref{sec:examples}. Furthermore, the theoretical results are illustrated with examples there. For the article to have a moderate length and good readability, we illustrate our results on some examples only, and give some hints for further applications. The general framework for the results obtained in Section \ref{sec:theory} allows us to directly identify many other classes of problems as subclasses of the framework discussed before. More precisely, we apply the general impulse control theory to optimal stopping problems, multiple optimal stopping problems, and optimal switching problems and obtain the corresponding results for these classes. Furthermore, we treat on example with a discontinuous cost structure explicitly and give some hints for the solution of impulse control problems for general L\'evy processes.

\section{General theory}\label{sec:theory}
On a probability space $(\Omega,\mathcal{F},\Pro)$ with a filtration $(\mathcal{F}_t)_{t\geq0}$ we consider a stochastic process $X=(X_t)_{t\geq 0}$ with values in $(E,\mathcal{B})$, where $E$ is a locally compact separable metric space and $\mathcal{B}$ denotes the Borel $\sigma$-algebra. We assume that $X$ has c\`adl\`ag paths and is quasi left-continuous (left-continuous over stopping times). Furthermore, we assume $X$ to be a strong Markov process with respect to the family $(\Pro_x)_{x\in E}$ of probability measures with a measurable time shift operator $\theta$. Without loss of generality, we can assume that the process $X$ is given on the canonical space and the time-shift acts as $\theta_t[(\omega_s)_{s\geq 0}]=(\omega_{t+s})_{s\geq 0}$.\\
%

Our set of strategies are \textit{impulse control strategies}; these are sequences $S=(\tau_n,\gamma_n)_{n\in\N}$. For general Markov processes, the definition of the controlled process with respect to $S$ is not immediate. We only give an intuitive explanation here and remind the reader of the formal definition in the Appendix \ref{appendix:impulse}. \\
Under the controlled measures $(\Pro^S_x)_{x\in E}$, between each two random times $\tau_{n-1}<\tau_n$, the process runs uncontrolled with the same dynamics as the original process. At each random time $\tau_n$ an impulse is exercised and the process is restarted at the new state $\gamma_n$. Here, $\tau_n$ is a stopping time for the process $X^n$ with only $n-1$ controls and $\gamma_n$ is measurable with respect to the corresponding pre-$\tau_n$ $\sigma$-algebra. \\
For jump processes $X$, the following technical problem has to be taken into account: By construction, the process $X$ has a jump due to the control that take place in time $\tau_n$. But furthermore, the uncontrolled process $X^n$ on $[\tau_{n-1},\tau_n]$ may also have a jump in time $\tau_n$ if $X^n$ does not have continuous sample paths. Therefore, $X^n_{\tau_n}\not=X_{\tau_n-}$ in general. For our further considerations, it will be important to consider the process $X^n$ also at time point $\tau_n$. Therefore, we write
\[X_{\tau_n,-}:=X^n_{\tau_n}\]
for the value of the process at $\tau_n$ if no control is exercised. Obviously, for continuous underlying processes, we have $X_{\tau_n,-}=X_{\tau_n-}$, which motivates this notation.\\
We furthermore assume that for each $x\in E$ the set $\mathcal{A}(x)\subseteq E$ is the set of possible states that the process $X$ may be shifted to from state $x$, that is $\gamma_n$ is $\mathcal{A}(X_{\tau_n,-})$-valued. At time point $\tau_n$, the controlled process is shifted to a point $\gamma_n$ and between two stopping times $\tau_n$ and $\tau_{n+1}$ the process behaves as the uncontrolled Markov process. Note that $X_{\tau_n}=X_{\tau_{n,-}}$ is allowed. We call this action a degenerated shift and assume (without loss of generality) that this is always allowed. This corresponds to the case that no control takes place. As usual, we call an impulse control strategy $S=(\tau_n,\gamma_n)_{n\in\N}$ \textit{admissible}, if $\tau_n\nearrow\infty$ for $n\nearrow\infty$. \\
Moreover, let $K:E\times E\rightarrow\R$ be a measurable function, the cost functional, fulfilling
\begin{equation}\label{eq:ass_K}
\E_x^S\left(\sum_{n=0}^\infty e^{-r\tau_n}K^-(X_{\tau_n,-},X_{\tau_n})\right)<\infty
\end{equation}
for all admissible impulse control strategies $S=(\tau_n,\gamma_n)_{n}$, where $K^-$ denotes the negative part of $K$. We often think of the case that the cost functional $K$ is non negative. In that case, the assumption obviously holds true. We want to remark that in our discussion the cost function does not only depend on the difference $x-y$ (assuming that $E$ is a vector space) as in many other treatments of impulse control problems, but is an arbitrary measurable function of two variables.\\
In the following we interpret $K(x,y)$ as the cost of shifting the process from state $x$ to state $y$. Therefore, it is reasonable to assume that $K(x,x)=0$ for all $x$. With this convention, it it no restriction to assume that all stopping times $\tau_n$ are finite a.s., and to assume $\tau_0=0,X_{\tau_0}=X_0$.\\
We furthermore fix a measurable function $f:E\rightarrow \R$ such that 
\[\E_x^S\int_0^\infty e^{-rs}|f(X_s)| ds<\infty,\;\;x\in E,\;S.\]
Since degenerated shifts are allowed, it particularly holds that
\[\E_x\int_0^\infty e^{-rs}|f(X_s)| ds<\infty,\;\;x\in E.\] 
Therefore, the $r$-resolvent $\ov f$ of $f$ for the uncontrolled process is well defined, that is
\[\ov f(x)=\E_x\int_0^\infty e^{-rs}f(X_s) ds,\;\;x\in E.\]
We consider the impulse control problem given by the following value function
\[v(x)=\sup_{S=(\tau_n,\gamma_n)_{n}}\E_x^S\left(\int_0^\infty e^{-rs}f(X_s)ds-\sum_{n=0}^\infty e^{-r\tau_n}K(X_{\tau_n,-},X_{\tau_n})\right),\;\;x\in E,\]
where the supremum is taken over all admissible impulse control strategies.\\
First, note that for each $S=(\tau_n,\gamma_n)_{n}$ the expectation
\[\E_x^S\left(\int_0^\infty e^{-rs}f(X_s)ds-\sum_{n=0}^\infty e^{-r\tau_n}K(X_{\tau_n,-},X_{\tau_n})\right)\]
is well-defined in $[-\infty,\infty)$ since the first summand is integrable and the second is integrable in $[-\infty,\infty)$ by assumption \eqref{eq:ass_K}. Since degenerated shifts are allowed, we obtain
\begin{equation}\label{eq:vgeqf}
v(x)\geq \E_x\int_0^\infty e^{-rs}f(X_s)ds=\ov{f}(x)>-\infty,
\end{equation}
so that $v>-\infty$. We furthermore assume that $v$ is finite.\\
We first rewrite the reward of this control problem to deal with the integral term.

\begin{lem}\label{lem:resolvent}
\begin{enumerate}[(i)]
\item For all finite stopping times $\tau$ and all\ $x\in E$ it holds that
\[\E_x\int_0^\tau e^{-rs}f(X_s)ds=-\E_x\left(e^{-r\tau}\ov{f}(X_\tau)\right)+\ov f(x).\]
\item  For each admissible impulse control strategy $S=(\tau_n,\gamma_n)_{n}$ and each\ $x\in E$ it holds that
\begin{align*}
&\E_x^S\left(\int_0^\infty e^{-rs}f(X_s)ds-\sum_{n=0}^\infty e^{-r\tau_n}K(X_{\tau_n,-},X_{\tau_n})\right)\\
=&\E_x^S\left(\sum_{n=1}^\infty e^{-r\tau_n}\left(\ov f(X_{\tau_n})-\ov f(X_{\tau_n,-})-K(X_{\tau_n,-},X_{\tau_n})\right)\right)+\ov f(x).
\end{align*}
\item For all\ $x\in E$
\[v(x)-\ov f(x)=\sup_{S=(\tau_n,\gamma_n)_{n}}\E_x^S\sum_{n=1}^\infty e^{-r\tau_n}\ov K(X_{\tau_n,-},X_{\tau_n}),\]
where 
\[\ov K(X_{\tau_n,-},X_{\tau_n})=\ov f(X_{\tau_n})-\ov f(X_{\tau_n,-})-K(X_{\tau_n,-},X_{\tau_n})\]
\end{enumerate}
\end{lem}
\begin{proof}
(i) immediately holds by the strong Markov property. \\
For (ii) note that the process $X$ runs uncontrolled between each two stopping times $\tau_{n-1},\tau_n$. Therefore, (i) yields
\[\E_x^S\int_{\tau_{n-1}}^{\tau_n}e^{-rs}f(X_s)ds=-\E_x^S\left(e^{-r\tau_n}\ov f(X_{\tau_n,-})-e^{-r\tau_{n-1}}\ov f(X_{\tau_{n-1}})\right).\]
We obtain
\begin{align*}
&\E_x^S\left(\int_0^\infty e^{-rs}f(X_s)ds-\sum_{n=0}^\infty e^{-r\tau_n}K(X_{\tau_n,-},X_{\tau_n})\right)\\
=&\E_x^S\left(\sum_{n=1}^\infty \int_{\tau_{n-1}}^{\tau_n}e^{-rs}f(X_s)ds-\sum_{n=0}^\infty e^{-r\tau_n}K(X_{\tau_n,-},X_{\tau_n})\right)\\
=&\E_x^S\left(\sum_{n=1}^\infty -(e^{-r\tau_n}\ov f(X_{\tau_n,-})-e^{-r\tau_{n-1}}\ov f(X_{\tau_{n-1}})) -\sum_{n=1}^\infty e^{-r\tau_n}K(X_{\tau_n,-},X_{\tau_n})\right)
\end{align*}
Now note that for each $k$
\[\sum_{n=1}^k -(e^{-r\tau_n}\ov f(X_{\tau_n,-})-e^{-r\tau_{n-1}}\ov f(X_{\tau_{n-1}}))=\sum_{n=1}^ke^{-r\tau_n}(\ov f(X_{\tau_n})-\ov f(X_{\tau_n,-}))-e^{-r\tau_k}\ov f(X_{\tau_k})+\ov f(x).\]
By the same argument as in (i) and dominated convergence it holds that
\[\E_x^S\left(e^{-r\tau_n}\ov{f}(X_{\tau_n})\right)=\E_x^S\int_0^{\infty} e^{-rs}f(X_s)ds-\E_x^S\int_0^{\tau_n} e^{-rs}f(X_s)ds\rightarrow0.\]
Therefore,
\begin{align*}
&\E_x^S\left(\int_0^\infty e^{-rs}f(X_s)ds-\sum_{n=0}^\infty e^{-r\tau_n}K(X_{\tau_n,-},X_{\tau_n})\right)\\
=&\E_x^S\left(\sum_{n=1}^\infty e^{-r\tau_n}\left(\ov f(X_{\tau_n})-\ov f(X_{\tau_n,-})-K(X_{\tau_n,-},X_{\tau_n})\right)\right)+\ov f(x).
\end{align*}
Since the summand $\ov f(x)$ is independent of $S$, taking the supremum over all $S$ gives (iii).
\end{proof}
Note that by (iii) of the previous Lemma, we could assume -- without loss of generality -- that $f=0$. But since an integral term arises in many problems of interest, we keep a general $f$ in the following. \\

Now, we introduce a set of superharmonic functions, that will be the main ingredient for our further considerations:
\[\mathcal{H}:=\{h:E\rightarrow \R| h\mbox{ is $r$-superharmonic, }h\geq 0,\;h+\ov f\geq M(h+\ov f)\},\]
where the maximum operator $M$ is given by 
\[Mw(x)=\sup_{y\in \mathcal{A}(x)\setminus\{x\}}(w(y)-K(x,y)),\]
with the convention $\sup\emptyset=-\infty$.

Next, we see that each function in $\mathcal{H}$ is an upper bound for the value function. Furthermore, we get lower bounds for some special functions.
\begin{prop}\label{prop:h_maj} Let $h:E\rightarrow \R$.
\begin{enumerate}[(i)]
\item If $h\in \mathcal{H}$, it holds that
\[v\leq h+\ov f.\] 
\item If $x\in E$ and $S=(\tau_n,\gamma_n)_{n}$ is an impulse control strategy such that 
\begin{equation}\label{eq:harmonic}
\E_x^S\left(e^{-r\tau_n}h(X_{\tau_n,-})\right)=\E_x^S\left(e^{-r\tau_{n-1}}h(X_{\tau_{n-1}})\right)\mbox{ for all }n\in\N,
\end{equation}
\[(h+\ov f)(X_{\tau_n})-K(X_{\tau_n,-},X_{\tau_n})\geq (h+\ov f)(X_{\tau_n,-})\;\;\;\Pro_x^S-a.s.,\]
and fulfilling the growth condition 
\[\E_x^Se^{-r\tau_n}h(X_{\tau_n,-})\rightarrow 0\mbox{ for }n\rightarrow\infty,\]
then
\[(h+\ov f)(x)\leq v(x).\]

\end{enumerate}
\end{prop}
\begin{proof}
Let $S=(\tau_n,\gamma_n)_{n}$ be an arbitrary admissible impulse control strategy and $x\in\ E$ such that
\[\E_x^S\left(\int_0^\infty e^{-rs}f(X_s)ds-\sum_{n=0}^\infty e^{-r\tau_n}K(X_{\tau_n,-},X_{\tau_n})\right)>-\infty.\]
Since $h\in\mathcal{H}$, by the optional sampling theorem for nonnegative supermartingales we obtain (keeping in mind that $X$ runs uncontrolled between $\tau_{n-1}$ and $\tau_n$ under $\E^S$)
\[\E_x^S\left(e^{-r\tau_n}h(X_{\tau_n,-})-e^{-r\tau_{n-1}}h(X_{\tau_{n-1}})\right)\leq 0.\]
Using this inequality and Lemma \ref{lem:resolvent} we get
\begin{align*}
&\E_x^S\left(\int_0^\infty e^{-rs}f(X_s)ds-\sum_{n=0}^\infty e^{-r\tau_n}K(X_{\tau_n,-},X_{\tau_n})\right)\\
=&\E_x^S\left(\sum_{n=1}^\infty e^{-r\tau_n}\left(\ov f(X_{\tau_n})-\ov f(X_{\tau_n,-})-K(X_{\tau_n,-},X_{\tau_n})\right)\right)+\ov f(x)\\
\leq&\E_x^S\left(\sum_{n=1}^\infty e^{-r\tau_n}\left((h+\ov f)(X_{\tau_n})-(h+\ov f)(X_{\tau_n,-})-K(X_{\tau_n,-},X_{\tau_n})\right)\right)+(h+\ov f)(x).
\end{align*}
Since $h+\ov f\geq M(h+\ov f)$ we obtain that 
\[(h+\ov f)(X_{\tau_n})-(h+\ov f)(X_{\tau_n,-})-K(X_{\tau_n,-},X_{\tau_n})\leq 0,\]
therefore
\begin{align*}
&\E_x^S\left(\int_0^\infty e^{-rs}f(X_s)ds-\sum_{n=0}^\infty e^{-r\tau_n}K(X_{\tau_n,-},X_{\tau_n})\right)\leq (h+\ov f)(x).
\end{align*}
Because $S$ was arbitrary, we see that $v(x)\leq (h+\ov f)(x)$, that is (i).\\
On the other hand, under the stated conditions we obtain (ii) by following the previous proof. 
\end{proof}

The previous proposition can be seen as a verification theorem. Indeed, it is a generalization (with less-explicit assumptions) of \cite[Theorem 6.2]{OS}. Now, we examine the structure of the solution more detailed:

\begin{satz}\label{thm:measurable}
Assume that $v$ is measurable.\\
Then $v-\ov f$ is the pointwise minimizer of $\Ha$, i.e. $v-\ov f\in \Ha$ and $v-\ov f\leq h$ for all $h\in \mathcal{H}$.
\end{satz}
\begin{proof}
By Proposition \ref{prop:h_maj} (i) it suffices to show that $v-\ov f\in \Ha$. Note that $v-\ov f\geq 0$ by \eqref{eq:vgeqf}. On the other hand, since immediate control is possible, $(v-\ov f)+\ov f=v\geq Mv=M((v-\ov f)+\ov f)$.\\
It remains to prove that $v-\ov f$ is $r$-superharmonic. Let\ $\sigma$ be a finite stopping time. By the measurability of $v$, we know that the following function is a random variable:
\[e^{-r\sigma}(v-\ov f)(X_\sigma)=\sup_{S=(\tau_n,\gamma_n)_{n}}e^{-r\sigma}\E_{X_\sigma}^S\sum_{n=1}^\infty e^{-r\tau_n}\ov K(X_{\tau_n,-},X_{\tau_n}),\]
where the equality holds by Lemma \ref{lem:resolvent}.
On the other hand, using the strong Markov property, almost surely we have for each admissible impulse control $S=(\tau_n,\gamma_n)_{n}$
\begin{align*}
&e^{-r\sigma}\E_{X_\sigma}^S\sum_{n=1}^\infty e^{-r\tau_n}\ov K(X_{\tau_n,-},X_{\tau_n})\\
=&\E_{x}^{S_\sigma}\left(\sum_{n=1}^\infty e^{-r\tau_{n,\sigma}}\ov K(X_{\tau_{n,\sigma}-},X_{\tau_{n,\sigma}})\Big|\mathcal{F}_\sigma\right),
\end{align*}
where $S_\sigma=(\tau_{n,\sigma},\gamma_{n,\sigma})$ is the time-shifted impulse control given by\ $\tau_{n,\sigma}=\sigma+\tau_n\circ\theta_\sigma,\gamma_{n,\sigma}=\gamma_n\circ\theta_\sigma$. This -- together with the measurability -- shows that $e^{-r\sigma}(v-\ov f)(X_\sigma)$ is the essential supremum of the set
\[\left\{\E_{x}^{S_\sigma}\left(\sum_{n=1}^\infty e^{-r\tau_{n,\sigma}}\ov K(X_{\tau_{n,\sigma}-},X_{\tau_{n,\sigma}})\Big|\mathcal{F}_\sigma\right):S\mbox{ impulse control}\right\}.\]
Following the line of arguments in \cite[p. 47]{PS}, it is easily seen that this set is directed upwards. By the standard properties of the essential supremum, there exists a sequence $(S_k)_{k\in\N}$ such that
\[\E_{x}^{S_{k,\sigma}}\left(\sum_{n=1}^\infty e^{-r\tau_{k,n,\sigma}}\ov K(X_{\tau_{k,n,\sigma}-},X_{\tau_{k,n,\sigma}})\Big|\mathcal{F}_\sigma\right)\nearrow e^{-r\sigma}(v-\ov f)(X_\sigma),\;\;k\nearrow\infty.\]
By the monotone convergence theorem we obtain
\[\E_x e^{-r\sigma}(v-\ov f)(X_\sigma)=\lim_{k\rightarrow\infty}\E_{x}^{S_{k,\sigma}}\left(\sum_{n=1}^\infty e^{-r\tau_{k,n,\sigma}}\ov K(X_{\tau_{k,n,\sigma}-},X_{\tau_{k,n,\sigma}})\right)\leq (v-\ov f)(x).\]
\end{proof}

The assumption that $v$ is measurable is natural in many situations of interest. For example, whenever the function
\[x\mapsto\E_x^S\left(\int_0^\infty e^{-rs}f(X_s)ds-\sum_{n=0}^\infty e^{-r\tau_n}K(X_{\tau_n,-},X_{\tau_n})\right)\]
is lower semicontinuous for each $S$, then $v$ is lower semicontinuous as a supremum of lower semicontinuous functions. Another sufficient condition is given in the following corollary.

\begin{koro}
Assume that there exists an optimal impulse control strategy $S=(\tau_n,\gamma_n)_n$.\\
Then $v-\ov f$ is the pointwise minimizer of $\Ha$.
\end{koro}

\begin{proof}
By Lemma \ref{lem:resolvent} we see that 
\[(v-\ov f)(x)=\E_x^S\sum_{n=1}^\infty e^{-r\tau_n}\ov K(X_{\tau_n,-},X_{\tau_n}).\]
Therefore, $v$ is measurable and the claim holds by the Theorem \ref{thm:measurable}.
\end{proof}

In the following, we construct an optimal admissible impulse control strategy $S=(\tau_n,\gamma_n)_{n}$, i.e. $S=(\tau_n,\gamma_n)_{n}$ is an admissible impulse control strategy and for all $x\in E$ it holds that
\[v(x)=\E_x^S\left(\int_0^\infty e^{-rs}f(X_s)ds-\sum_{n=0}^\infty e^{-r\tau_n}K(X_{\tau_n,-},X_{\tau_n})\right).\]
To this end, we will often assume the following weak form of the triangle inequality for the cost function:
For all $x\in E$ and $y\in \mathcal{A}(x)\setminus \{x\}$, there exists some $\epsilon>0$ such that for all $z\in \mathcal{A}(x)\cap \mathcal{A}(y)$
\begin{equation}\label{eq:triangle}
K(x,y)+K(y,z)\geq K(x,z)+\epsilon.
\end{equation}
Note that this is a natural assumption for optimal impulse control problems, where often two types of costs are assumed: Fixed costs and proportional costs. The proportional costs naturally fulfill the standard triangle inequality. Now, since the fixed costs have to be added, the extra summand  $\epsilon$ is natural. For example, in the survey article \cite{K} the cost structure in $\R^d$ was assumed to have the form $K(x,y)=|x-y|+K$ for some $K>0$, where the assumption \eqref{eq:triangle} is obviously fulfilled. Furthermore, we often assume that for all $x\in E$
\begin{equation}\label{eq:trading}
\mathcal{A}(y)\subseteq \mathcal{A}(x)\mbox{  for all }y\in\mathcal{A}(x).
\end{equation}
In other words, \eqref{eq:trading} means that if it is possible to shift the process from state $x$ to state $y$ and from state $y$ to state $z$, then it is also possible to shift the process from state $x$ to state $z$ directly. The natural conditions \eqref{eq:triangle} and \eqref{eq:trading} guarantee that if it is rational to trade from $x$ to $y$, then no immediate trading in $y$ is rational:

\begin{prop}\label{prop:continuation}
Assume \eqref{eq:triangle} and \eqref{eq:trading}. Let $\ov v:E\rightarrow \R$ be a measurable and $x,y\in E,\;x\not=y,$ such that
 \[M\ov v(x)=\ov v(y)-K(x,y).\] Then,
\[\ov v(y)>M\ov v(y).\]
\end{prop}

\begin{proof}
Choose $\epsilon >0$ as in \eqref{eq:triangle} and let\ $z\in \mathcal{A}(y)$. By \eqref{eq:trading} we have $z\in\mathcal{A}(x)$. We obtain
\begin{align*}
\ov v(z)-K(y,z)&\leq \ov v(z)-(K(x,z)-K(x,y)+\epsilon)\\
&\leq  M\ov v(x)+K(x,y)-\epsilon\\
&=\ov v(y)-K(x,y)+K(x,y)-\epsilon\\
&=\ov v(y)-\epsilon,
\end{align*}
hence $\ov v(y)\geq \ov v(z)-K(y,z)+\epsilon$. Taking supremum over all $z$ yields 
\[\ov v(y)\geq M\ov v(y)+\epsilon>M\ov v(y).\]
\end{proof}

Now, we come to the second main result of this section, that is a theorem that guarantees the existence of an optimal impulse control strategy. Furthermore, the optimal strategy is described in terms of the pointwise minimum of $\mathcal{H}$ under natural assumptions. The advantage of this theorem compared to the previous results is that it is stated in term of the minimizer of $\Ha$ and not in terms of the (unknown) value function $v$. Therefore, no (direct) regularity assumptions on $v$ are needed, that are often hard to establish.

\begin{satz}\label{thm:minimizer}
Assume that \eqref{eq:triangle} and  and \eqref{eq:trading} hold true and that $\ov f$ is nonnegative and lower semicontinuous.\\
Assume that $h$ is a pointwise minimizer in $\Ha$, that $\ov v:=h+\ov f$ is lower semicontinuous (lsc), and $M\ov v$ is upper semicontinuous (usc), that fulfills the integrability condition
\begin{equation}\label{eq:integrability}
\E_x\sup_{t\geq0}e^{-rt}|M(\ov v-\ov f)(X_t)|<\infty\mbox{ for all }x\in E.
\end{equation}
Furthermore, assume that the stopping time 
\[\tau_{\ov S}:=\inf\{t\geq 0:X_t\in \ov S\},\;\;\;\ov S=\{x\in\ E:\ov v(x)=M\ov v(x)\},\]
is finite $\Pro_x$-a.s. for all $x\in E$, and that there exists a measurable function $\phi:\ov S\rightarrow E$ such that for each $x\in \ov S$
\[M\ov v(x)=\ov v(\phi(x))-K(x,\phi(x)),\]
and let the impulse control strategy $S$ given by
\begin{align*}
&\tau_0=0,\;\;\gamma_0=x\;\Pro_x-\mbox{a.s. for all }x\in E,\\
&\tau_n=\inf\{t>\tau_{n-1}:\ov v(X_t)=M\ov v(X_t)\},\\
&\gamma_n=\phi(X_{\tau_n,-})
\end{align*}
be admissible and $\E_x^S e^{-r\tau_n}h(X_{\tau_n})\rightarrow 0$.\\[.2cm]
Then it holds that 
\[v=\ov v\] 
and $S$ is an optimal admissible impulse control.
\end{satz}

\begin{proof}
\begin{enumerate}[(a)]
\item For $\lambda\in(0,1)$ write 
\[S_\lambda=\{x\in E: \lambda \ov v(x)\leq M\ov v(x)\}\mbox{ and }C_\lambda=S_\lambda^c.\]
Since $\ov v$ is lsc and $M\ov v$ is usc, we see that $S_\lambda$ is a closed set and 
 \[S_\lambda\searrow \ov S:=\{x\in E: \ov v(x)= M\ov v(x)\},\;\; \lambda\nearrow 1.\]
Write $\tau_\lambda:=\inf\{t\geq 0:\ X_t\in S_\lambda\}$.  Since $\tau_{\ov S}<\infty$ $\Pro_x$-a.s. for all $x\in E$ and since $ \ov S\subseteq S_\lambda$ we obtain that $\tau_\lambda$ is a.s. finite under all measures $\Pro_x,x\in E$. 
\item  Fix $\lambda<1$. We write 
\[g(x):=\E_xe^{-r\tau_\lambda}(\ov v-\ov f)(X_{\tau_\lambda})\mbox{, $x\in E$.}\] 

Now we show $\tilde v-\ov f\in \mathcal{H}$, where  $\tilde v(x):=\lambda\ov v(x)+(1-\lambda)(g+\ov f)(x)$. First we show that for all $x\in E$ it holds that 
\begin{equation}\label{eq:M_tilde_v_gleichung}
M\tilde v(x)\leq \tilde v(x).
\end{equation}
By noting that $M$ is a convex operator (see also Subsection \ref{subsec:convex}), we have
\begin{align*}
M\tilde v(x)&\leq \lambda M\ov v(x)+(1-\lambda)M(g+\ov f)(x)\\
&= \lambda M\ov v(x)+(1-\lambda)\sup_{y\in \mathcal A(x)}[\E_y e^{-r\tau_\lambda}(\ov v-\ov f)(X_{\tau_\lambda})+\ov f(y)-K(x,y)]\\
&\leq \lambda M\ov v(x)+(1-\lambda)\sup_{y\in \mathcal A(x)}[(\ov v-\ov f)(y)+\ov f(y)-K(x,y)]\\
&= \lambda M\ov v(x)+(1-\lambda)M\ov v(x)=M\ov v(x).
\end{align*}
If $x\in S_\lambda$, by the previous inequality and the definition of $\ov v$ it holds that
\[M\tilde v(x)\leq M\ov v(x)\leq \ov v(x)=\lambda\ov v(x)+(1-\lambda)\ov v(x)=\lambda\ov v(x)+(1-\lambda)(g+\ov f)(x)=\tilde{v}(x),\]
where we used that $\tau_\lambda=0$ $\Pro_x$-a.s. for $x\in S_\lambda$, which implies 
\[g(x)+\ov f(x)=\E_xe^{-r\tau_\lambda}(\ov v-\ov f)(X_{\tau_\lambda})+\ov f(x)=\ov v(x)\mbox{ for $x\in S_\lambda$}.\]
 On the other hand, for $x\in C_\lambda$ we have $M\ov v(x)\leq \lambda \ov v(x)$. Hence,
\begin{align*}
M\tilde v(x)&\leq M\ov v(x)\leq \lambda\ov v(x).
\end{align*}
To obtain \eqref{eq:M_tilde_v_gleichung}, it remains to be proved that $\lambda\ov v(x)\leq \tilde v(x)=\lambda\ov v(x)+(1-\lambda)(g+\ov f)(x)$ for $x\in C_\lambda$. This holds since $h,\ov f\geq0$, which implies
\[(g+\ov f)(x)=\E_xe^{-r\tau_\lambda}(\ov v-\ov f)(X_{\tau_\lambda})+\ov f(x)=\E_xe^{-r\tau_\lambda}h(X_{\tau_\lambda})+\ov f(x)\geq 0.\]
 Since $\ov v-\ov f$ is $r$-superharmonic, by the general theory of superharmonic functions we know that so is  $g:x\mapsto \E_xe^{-r\tau_\lambda}(\ov v-\ov f)(X_{\tau_\lambda})$. Therefore, so is $\tilde v-\ov f$. Hence, we have proved that $\tilde v-\ov f\in\mathcal H$.\\
 By the minimality property of $h=\ov v-\ov f$ we obtain  
\begin{align*}
(\ov v-\ov f)(x)&\leq (\tilde v-\ov f)(x)=\lambda(\ov v-\ov f)(x)+(1-\lambda)g(x)\\
&=\lambda(\ov v-\ov f)(x)+(1-\lambda)\E_xe^{-r\tau_\lambda}(\ov v-\ov f)(X_{\tau_\lambda}),
\end{align*}
i.e. $(\ov v-\ov f)(x)\leq \E_xe^{-r\tau_\lambda}(\ov v-\ov f)(X_{\tau_\lambda}).$
Keeping in mind that $\ov v- \ov f$ is $r$-superharmonic, we obtain
\[h(x)=(\ov v-\ov f)(x)=\E_xe^{-r\tau_\lambda}(\ov v-\ov f)(X_{\tau_\lambda})=\E_xe^{-r\tau_\lambda}h(X_{\tau_\lambda}). \]
\item Since $\tau_\lambda$ is monotonically increasing in $\lambda$, $\tau:=\lim_{\lambda\rightarrow1}\tau_\lambda$ exists and $\tau\leq \tau_{\ov S }$. As $X$ is quasi left-continuous, it holds that $\lim_{\lambda\nearrow 1}X_{\tau_\lambda}=X_\tau$. Because of the semicontinuity of $\ov v$ and $M\ov v$ and because $\ov v(X_{\tau_\lambda})\leq \frac{1}{\lambda}M\ov v(X_{\tau_\lambda})$ for $\lambda\nearrow 1$, we have that $\ov v(X_\tau)=M\ov v(X_\tau)$. Therefore, $\tau_{\ov S}\leq \tau$, i.e. $\tau_{\ov S}= \tau$.
Using dominated concvergence and the usc of $M\ov v-\ov f$ we obtain
\begin{align*}
h(x)&=\liminf_{\lambda\nearrow 1} \E_xe^{-r\tau_\lambda}h(X_{\tau_\lambda})\leq \liminf_{\lambda\nearrow 1} \frac{1}{\lambda}\E_xe^{-r\tau_\lambda}(M\ov v-\ov f)(X_{\tau_\lambda})\\
&\leq \E_xe^{-r\tau}(M\ov v-\ov f)(X_{\tau})\leq \E_x e^{-r\tau}h(X_{\tau})=\E_x e^{-r\tau_{\ov S}}h(X_{\tau_{\ov S}}).
\end{align*}
Since $h$ is $r$-superharmonic we obtain
\begin{equation}\label{eq:harm}
h(x)=\E_x e^{-r\tau_{\ov S}}h(X_{\tau_{\ov S}}).
\end{equation}
\item Note that $\ov v(X_{\tau_{\ov S}})=M\ov v(X_{\tau_{\ov S}})$. Furthermore, \eqref{eq:harm} yields that condition \eqref{eq:harmonic} from Proposition \ref{prop:h_maj} is fulfilled for the impulse control strategy given above. This strategy therefore fulfills the requirements of Proposition \ref{prop:h_maj} (ii). To see that it is indeed an impulse control strategy, note that $\tau_n<\tau_{n+1}$ is fulfilled by Proposition \ref{prop:continuation}. We obtain $\ov v(x)=v(x)$ and the optimality of the impulse control strategy. 
\end{enumerate}
\end{proof}

\begin{anmerk}
Note that the condition \eqref{eq:integrability} is a natural condition to guarantee that the value function is finite. For the corresponding equation for optimal stopping problems, see the connection in Subsection \ref{subsec:OS}.
\end{anmerk}

\begin{anmerk}
The assumption that $\ov f$ is nonnegative in the previous theorem can be weakened by assuming that the function $f$ is bounded below by some constant $c<0$. Then for all $x\in E$ it holds that
\begin{align*}
v(x)&=\sup_{S=(\tau_n,\gamma_n)_{n}}\E_x^S\left(\int_0^\infty e^{-rs}f(X_s)ds-\sum_{n=0}^\infty e^{-r\tau_n}K(X_{\tau_n,-},X_{\tau_n})\right)\\
&=\sup_{S=(\tau_n,\gamma_n)_{n}}\E_x^S\left(\int_0^\infty e^{-rs}(f-c)(X_s)ds-\sum_{n=0}^\infty e^{-r\tau_n}K(X_{\tau_n,-},X_{\tau_n})\right)+c/r
\end{align*}
and $f-c\geq 0$, which yields that the resolvent of ${f-c}$ is nonnegative and Theorem \ref{thm:minimizer} can be applied to the problem for $f-c$. 
\end{anmerk}

\section{Discussion and examples}\label{sec:examples}

\subsection{Impulse control as a convex optimization problem}\label{subsec:convex}
It is often convenient to consider optimal stopping problems as linear programming problems, see for example \cite{CS} for a discussion in a general setting. In the same line, the previous discussion shows that impulse control problems may be seen as convex optimization problems. Indeed, for a fixed state $x_0\in E$, we have seen in Theorem \ref{thm:measurable} and \ref{thm:minimizer} that -- under some natural conditions -- the value $v(x_0)$ is given as the the minimum of $h(x_0)$, where the minimum is taken over all $h\in \mathcal{H}$. Note that $\mathcal{H}$ is indeed convex since for all $\lambda\in[0,1],h_1,h_2\in\mathcal{H}$ for $h:=\lambda h_1+(1-\lambda)h_2\geq 0$ it holds that $h$ is $r$-superharmonic and for all $x\in E$
\begin{align*}
(h+\ov f)(x)&=\lambda (h_1+\ov f)(x)+(1-\lambda)(h_2+\ov f)(x)\\
&\geq \lambda M(h_1+\ov f)(x)+(1-\lambda)M(h_2+\ov f)(x)\\
&\geq M(\lambda (h_1+\ov f)+(1-\lambda)(h_2+\ov f))(x)\\
&=M(h+\ov f)(x),
\end{align*}
so that $h\in \mathcal{H}$. Therefore, the impulse control problem can be seen as the following convex programming problem:
\begin{align*}
\min_{h\mbox{ superharm.}} &\;\;\;\;h(x_0)\\
\mbox{subj. to}&\;\;\;\;(h+\ov f)(x)\geq M(h+\ov f)(x)\mbox{ for all }x\in E.
\end{align*}

\subsection{Connection to quasi-variational inequalities}
Now, we can identify the value function as a solution to the corresponding quasi-variational inequality under appropriate regularity conditions as follows: As described in the previous section, we can identify $h:=v-\ov f$ as the smallest $r$-superharmonic function with 
\begin{equation}\label{eq:max_qvi}
M(h+\ov f)\leq h+\ov f.
\end{equation} 
Now, we assume that $v$ is regular enough to apply the generator (or Dynkin operator) $A$ of $X$. Then we obtain
\[0\geq (A-r)(v-\ov f)(x)=(A-r)v(x)+f(x),\]
i.e. $(A-r)v+f\leq 0$. Moreover, by \eqref{eq:max_qvi} it holds that $Mv\leq v$. By the considerations leading to \eqref{eq:harm}, it is furthermore clear that $v-\ov f$ is $r$-harmonic on $\{Mv<v\}$, i.e. $(A-r)v+f=0$ on this set. We obtain that $v$ is a solution to the quasi-variational inequality
\[\max\{(A-r)v+f,Mv-v\}=0.\]
But note that for our approach, no further regularity assumptions on $v$ are needed, see also Example \ref{subsec:discont} below.


\subsection{One-dimensional diffusion processes}
For the explicit applicability of the theory, it is of interest to have a more explicit characterization of the $r$-superharmonic functions for the process $X$. Such a characterization is well-known for regular one-dimensional diffusion processes $X$ with absorbing or natural boundaries. In this case, a function $h$ is $r$-superharmonic if and only if $\frac{h}{\phi}$ is $\frac{\psi}{\phi}$-concave, where $\phi,\psi$ denote the increasing resp. decreasing fundamental $r$-harmonic functions. We refer to \cite{DK} for a recent treatment. One could say that the nonnegative $r$-superharmonic functions can be characterized as the concave functions in a transformed space. Therefore, one can characterize the value function geometrically as the smallest nonnegative extended concave function that fulfills $M(h+\ov f)\leq h+\ov f$. The main difficulty -- compared to the optimal stopping problem for one-dimensional diffusion processes -- is that the condition $M(h+\ov f)\leq h+\ov f$ is a nonlocal condition in general, since $M$ is a nonlocal operator. Under assumptions that simplify the operator $M$ in a suitable way, one can be hopeful to solve the problem geometrically. In some special situations, this idea was carried out, see \cite{E}. The main structural results obtained there -- in our notation -- is the following (see \cite[Proposition 3.1]{E}):
\begin{prop}\label{prop:Egami}
Under the assumptions stated in \cite{E}, $v-\ov f$ is the smallest function $h\geq 0$ such that
\[(h+\ov f)(x)=\sup_{\tau}\E_xe^{-r\tau}M(h+\ov f)(X_\tau).\]
\end{prop}
This result can also be seen as the main ingredient used in \cite{A}, see equation (2.8) there. By the general theory of optimal stopping, the previous fact can be stated in the following form: Writing
\begin{align*}
\tilde{\mathcal{H}}:=\{h:E\rightarrow \R| &h\mbox{ is $r$-superharmonic, }h\geq 0,\;h+\ov f\geq M(h+\ov f),\\
&\forall\mbox{$w$ $r$-superharm. with $w+\ov f\geq M(h+\ov f)$}:w\geq h\},
\end{align*}
Proposition \ref{prop:Egami} states that $v-\ov f$ is a pointwise minimizer of $\tilde{\mathcal{H}}$. Obviously, $\tilde{\mathcal{H}}\subseteq {\mathcal{H}}$. On the other hand, it is easily seen -- using \eqref{eq:harm} -- that $v\in \tilde{\mathcal{H}}$, so that Proposition \ref{prop:Egami} can be obtained from our general theory in the previous section. 
In all the examples discussed in \cite{A}, \cite{E}, and \cite{AL}, the assumptions were chosen such that the set $S=\{x\in E:v(x)=Mv(x)\}$ turns out to be essentially one-sided and the process is shifted back to one special point. Next, we discuss a (depending on the parameter) one- or two-sided problem, that can be dealt with using our approach:

\subsection{Example: Discontinuous costs}\label{subsec:discont}
Now, we treat one example in detail, since it deliver insights into the use of $r$-superharmonic functions for the solution of impulse control problems. For simplicity, we consider a standard Brownian motion $X$ as an underlying process on $E=\R$. Let $f\equiv 0$ and assume that at each intervention we can shift the process to the state $0$ or go on, that is $\mathcal{A}(x)=\{0,x\}$ for all $x\in \R$. If we stop at a state $x\geq 1$, we receive an amount of 1, and we have to pay costs of 1 for $x<1$. In our notation, we have $K(x,0)=-1$ for $x\geq 1$ and $K(x,0)=1$ for $x<1$ ($x\not=0$). Note that the cost functional is discontinuous at $x=1$, which is not easy to handle for the ordinary approaches to impulse control. Nonetheless, the formulation using superharmonic function can deal with this problem immediately:\\
Now, we discuss how to construct a pointwise minimizer in $\mathcal{H}$. First, note that 
\begin{equation}\label{eq:M_brownian}
Mh(x)=h(0)-K(x,0)=\begin{cases}h(0)+1&,\;\;x\geq 1,\\ h(0)-1&,\;\;x<1.\end{cases}
\end{equation}
For $x\geq 1$ it seems to be reasonable to stop immediately and receive the reward. Therefore, we make the Ansatz $Mh(x)=h(x)$ for $x\geq 1$, i.e. $h(x)=h(0)+1$. For $x<1$, it is not obvious if it is reasonable to shift the process to state 0. Indeed, it depends on the discounting parameter $r$. We distinguish two cases:\\[0.2cm]

\begin{figure}[h]
\begin{minipage}{0.5\textwidth}
\begin{center}
\includegraphics[height=5cm]{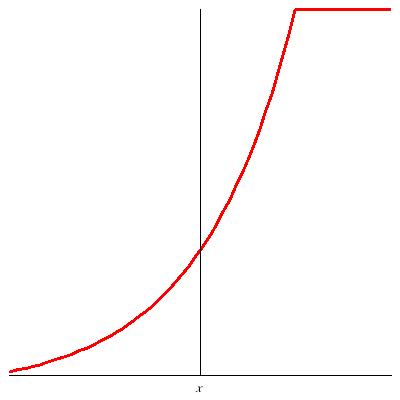}
\caption{Value function for case 1}\label{fig:v1}
\end{center}
\end{minipage}
\hfill
\begin{minipage}{0.5\textwidth}
\begin{center}
\includegraphics[height=5cm]{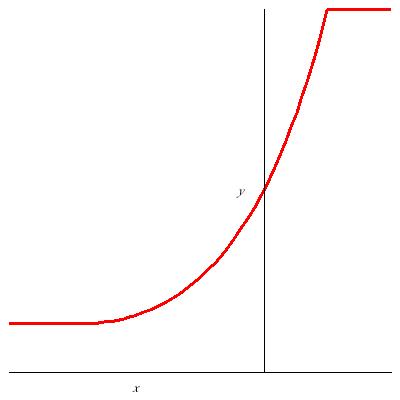}
\caption{Value function for case 2}\label{fig:v2}
\end{center}
\end{minipage}
\end{figure}

\textit{1. case}: $e^\beta\geq 2$, where $\beta=\sqrt{2r}$, i.e. the discounting factor is sufficiently high. We make the Ansatz, that $h$ is $r$-harmonic on $(-\infty,1]$. Furthermore, it seems to be reasonable, that $h$ is non-increasing on that interval. The general theory of $r$-harmonic functions yields that $h(x)=\lambda e^{\beta x}$ with $\beta$ as above for some $\lambda>0$. We see that $\lambda=h(0)$, so that we consider
\[h(x)=\begin{cases}
\lambda e^{\beta x},&\;\;x<1,\\
\lambda+1,&\;\;x\geq 1,
\end{cases}\]
see Figure \ref{fig:v1}. Since $r$-superharmonic functions for a standard Brownian motion are continuous, we should have $\lambda e^{\beta}=\lambda+1$, i.e. 
\[\lambda=\frac{1}{e^\beta-1}.\]
Now, we check that $h\in \mathcal{H}$. Obviously, $h\geq 0$. To see that $h$ is $r$-superharmonic, we recognize that $h=\min\{h_1,h_2\},$ where $h_1(x)=\lambda e^{\beta x} $, $h_2(x):=\lambda+1$. Since $h_1$ and $h_2$ are obviously $r$-superharmonic, so is $h$. We want to remark, that $h$ is not smooth at $x=1$; in particular it is not smooth enough to apply It\^o's formula (in its standard form without local time), as required in most verification theorems for impulse control. It remains to be checked that $Mh(x)\leq h(x)$. Keeping \eqref{eq:M_brownian} in mind, this is trivial for $x\geq 1$. For $x<1$, we have $Mh(x)=\lambda-1$. A short calculation yields that $Mh(x)\leq h(x)$ iff $e^\beta\geq 2$, which is the assumption above. In that case, we indeed have $Mh(x)< h(x)$ for all $x<1$. We have proved that $h\in \mathcal{H}$. \\
Using Proposition \ref{prop:h_maj} (i), we obtain that $h\geq v$. Inspired by Theorem \ref{thm:minimizer}, we define
\begin{align*}
&\tau_n=\inf\{t>\tau_{n-1}:h(X_t)=Mh(X_t)\}=\inf\{t>\tau_{n-1}:X_t\geq1\},\\
&\gamma_n=0.
\end{align*}
This is obviously an admissible impulse control strategy. Since $h$ is $r$-harmonic on $(-\infty,1)$, it fulfills the requirements of Proposition \ref{prop:h_maj} (ii) and we obtain that $h\geq v$ and the impulse control strategy mentioned above is optimal.\\[0.2cm]
\textit{2. case}: $e^\beta< 2$. In this case it turns out to be optimal to shift back the process to 0, whenever $X$ is below a threshold $x^*$, $x^*<0$ to be found. Analogously to the discussion above, we make the Ansatz
\[h(x)=\begin{cases}
h(0)-1,&\;\;x\leq x^*,\\
\lambda_1 e^{\beta x}+\lambda_2 e^{-\beta x},&\;\;x\in(x^*,1),\\
h(0)+1,&\;\;x\geq 1,
\end{cases}\]
see Figure \ref{fig:v2}. We find the unknown parameters $h(0),\lambda_1,\lambda_2,x^*$ via the conditions
\begin{align*}
\lambda_1 +\lambda_2=h(0),&\;\;\;\;\lambda_1 e^{\beta}+\lambda_2 e^{-\beta}=h(0)+1\\
\lambda_1 e^{\beta x^*}+\lambda_2 e^{-\beta x^*}=h(0)-1,&\;\;\;\;\beta\lambda_1 e^{\beta x^*}-\beta \lambda_2 e^{-\beta x^*}=0.
\end{align*}
Indeed, it is not hard to check, that the parameters are uniquely determined by this four equations under the assumption $e^\beta< 2$. The first two conditions are analogously to the 1. case, the second two guarantee that $h$ is smooth at $x^*$, which leads to the conclusion that 
\[h_1(x):=\begin{cases}
h(0)-1,&\;\;x\leq x^*,\\
\lambda_1 e^{\beta x}+\lambda_2 e^{-\beta x},&\;\;x\in(x^*,\infty),
\end{cases}\]
is $r$-superharmonic. Therefore, so is $h$ with the same argument as in the 1. case. Again, we obtain $h\in \mathcal{H}$ and using Proposition \ref{prop:h_maj} we see that $v=h$, and
\begin{align*}
&\tau_n=\inf\{t>\tau_{n-1}:X_t\not\in[x^*,1]\},\\
&\gamma_n=0
\end{align*}
is optimal.

\subsection{Connection to the ordinary theory of optimal stopping}\label{subsec:OS}
The line of argument given above is inspired by the treatment of optimal stopping problems, as presented for example in \cite{PS}. Now, we want to discuss how to find the optimal stopping problems as a subclass in the class of optimal impulse control problems. In the setting above, assume that $E$ contains a grave $\partial$, that is never reached by the uncontrolled process when started in $E\setminus \{\partial\}$. If the process reaches $\partial$, it stays there forever. Then we assume $\mathcal{A}(x)=\{ x,\partial\}$ for $x\not=\partial$ and $\mathcal{A}(\partial)=\{\partial\}$. Note that condition \eqref{eq:trading} is obviously fulfilled. Furthermore, write $g(x):=-K(x,\partial)$ for all $x$ and assume $f(\partial)=0$. For each admissible impulse control $S=(\tau_n,\gamma_n)$ (when ignoring the trivial case $X_{\tau_n}=X_{\tau_{n,-}}$), it holds that $\gamma_n=\partial$ for all $n$. Therefore, 
\[\E_x^S\left(\int_0^\infty e^{-rs}f(X_s)ds-\sum_{n=0}^\infty e^{-r\tau_n}K(X_{\tau_n,-},X_{\tau_n})\right)=\E_x\left(\int_0^{\tau_1} e^{-rs}f(X_s)ds+e^{-r \tau_1}g(X_{\tau_1-})\right).\]
We obtain that we are indeed faced with an ordinary optimal stopping problem with discounting and an integral term. Hence, we can consider this class of optimal stopping problems as a subclass of the impulse control problems. Letting $\ov f\equiv 0$ for simplicity and noting that $h(\partial)=0$ for all $r$-superharmonic functions $h$, we see that the condition
\[M(h+\ov f)(x)\leq h+\ov f\]
becomes
\[g(x)\leq h(x),\]
i.e. $h$ majorizes $g$. This corresponds to the well-known results for ordinary optimal stopping problems. Note that condition \eqref{eq:integrability} in Theorem \ref{thm:minimizer} boils down to the standard integrability condition
\[\E_x\sup_{t\geq 0}e^{-rt}|g(X_t)|<\infty.\]

\subsection{Multiple stopping problem with random refraction period}
As a generalization of optimal stopping problems, we now consider the following class of multiple stopping problems: In the last years, a theory was developed for solving multiple stopping problems inspired by applications to swing options in the energy market, see \cite{Car}, \cite{Touzi}, and \cite{CIJ}. In a Markov process setting, these are problems of the following form:
\[\sup_{\sigma_1,...,\sigma_k}\E_x\sum_{i=1}^ke^{-r\sigma_i}g(Y_{\sigma_i}),\]
where $Y$ is a strong Markov process with state space $\X$, $g$ is a measurable function $\geq 0$, and the supremum is taken over all stopping times $\sigma_1,...,\sigma_k$, where it is assumed that between each two exercises, there is a refraction period of deterministic length $\delta>0$, that is $\sigma_{i+1}\leq \sigma_i+\delta$ for all $i<k$. The main theoretical result was that this problem can be reduced to a sequence of $n$ ordinary optimal stopping problems, see \cite{Touzi} and \cite{CIJ}. \\
Now, we will show that this result can also be immediately obtained using the theory developed before. To this end, we introduce a new Markov $X$ with state space
\[E=\left(\bigcup_{i=1}^k\X\times\{i\}\right)\cup\left(\bigcup_{i=2}^k\X\times[0,\delta]\times\{i\}\right)\cup\{\partial\}\]
as follows: 
\begin{itemize}
\item $\partial$ is an absorbing state.
\item Started in a point $(x,i)\in\X\times \{i\},\;i\geq1,$ the process does not leave $\X\times \{i\}$ and has the same dynamics on this space as $Y$ on $\X$. 
\item Started in a point $(x,s,i)\in\X\times [0,\delta]\times\{i\},\;i\geq2$, the process is given by $(x,s+r,i)$ for $r\in[0,\delta-s)$ and is then restarted in $\X\times\{i-1\}$ with initial distribution $\Pro_x(X_\delta\in \cdot)$. 
\end{itemize}
For this process, we specify an impulse control problem as follows:
\begin{align*}
\A(z)&=\{z\}\mbox{ for all }z\in\left(\bigcup_{i=2}^k\X\times[0,\delta]\times\{i\}\right)\cup\{\partial\},\\
\A(x,i)&=\{(x,i),(x,0,i)\}\mbox{ for all }(x,i)\in\X\times\{i\},i\geq2,\\
\A(x,1)&=\{(x,1),\partial\}\mbox{ for all }(x,1)\in\X\times\{1\},\\
K((x,i),(x,0,i))&=-g(x)=K((x,1),\partial)\mbox{ for all }x\in \X,i\geq2,\\
f&=0.
\end{align*}
By the construction of the controlled process, we can identify each impulse control strategy $(\tau_n,\gamma_n)_{n\in\N}$ for a starting state in $\X\times\{k\}$ with a sequence of stopping times $\sigma_1,...,\sigma_k$ with $\sigma_{i+1}\leq \sigma_i+\delta$ for all $i<k$. Therefore, the multiple stopping problem can be identified with the impulse control problem described above. Let $v$ denote the value function of the impulse control problem. Then obviously
\[v(\partial)=0.\]
For the maximum operator, it holds that
\[Mw(x,1)=w(\partial)-K((x,1),\partial)=w(\partial)+g(x),\quad x\in\X,\]
Therefore, on $\X\times\{1\}$ $v$ can be found as the smallest $r$-superharmonic majorant of $g(x)$, this is $v$ is the value function of the optimal ordinary stopping problem for $g$. Now, write
\[h_{\delta,1}(x):=e^{-r\delta}\E_xv(Y_\delta,1)\left(=v(x,2,0)\right),\quad x\in\X.\]
Using this notation, we obtain on $\X\times\{2\}$ that the maximum operator is given by
\[Mw(x,2)=w(x,2,0)+g(x)=g(x)+e^{-\delta t}\E_xw(Y_\delta,1).\]
We obtain that on $\X\times \{2\}$ the value function $v$ is the smallest $r$-superharmonic majorant of the function
\[g+h_{\delta,1},\]
i.e. $v$ is found to be the value function of the ordinary stopping problem with reward function $g+h_{\delta,1}$. Using induction, we obtain the same result for each $k$ and we see that the value function can be found by solving a sequence of ordinary optimal stopping problems. 

\subsection{Optimal switching problems}
One of the most prominent class of impulse control problems is given by optimal switching problems, see e.g. \cite{BL}, \cite{LB}, and \cite{BE}. Using our notations, the problem can be stated as follows:\\
For two Markov processes $X^{(0)},X^{(1)}$ with joint state space $\hat E$ consider the space
\[E=\{0,1\}\times \hat E\]
and the stochastic process $X$ with (uncontrolled) distribution as those of $(i,X^{(i)})$ when started in $(i,x)\in\{i\}\times \hat E,\;i=0,1.$ The set of possible controls is then given by
\[\mathcal{A}(i,x)=\{(i,x),(i-1,x)\},\]
that is, the decision maker can control the distribution of the underlying process. Therefore, the maximum operator is given by 
\[Mh(i,x)=h(1-i,x)-k_i(x),\;i=0,1,\;x\in\hat E,\]
where we write $k_i(x)=K((i,x),(1-i,x))$. Therefore, the transformed value function $v-\ov f$ can be characterized as the smallest $r$-superharmonic function $h\geq 0$ that fulfills
\[(h+\ov f)(i,x)\geq (h+\ov f)(1-i,x)-k_i(x),\;i=0,1,\;x\in\hat E,\]
which can be interpreted as a coupled system of two optimal stopping problems.

\subsection{Application to general L\'evy processes}\label{subsec:levy}
Now, we discuss the case of a general L\'evy process $X$ on $E=\R$ to illustrate that the general theory leads to useful results in particularly interesting cases. In this generality, there is no hope to find the solution of the optimal impulse control problem explicitly in greater generality. Nonetheless, we want to describe the structure of the value function, that may be useful as an Ansatz in many concrete situations of interest. The main tool for such a representation is the general integral representation of $r$-superharmonic/excessive functions using the Riesz representation theorem. \\
We concentrate on the particularly interesting case that an optimal impulse control strategy is a constant-boundary strategy, that is there exist $a<\alpha\leq \beta<b$ such that the impulse control strategy $S$ given by
\begin{align*}
\tau_n&=\inf\{t>\tau_{n-1}:X_t\not\in(a,b)\}\\
\gamma_n&=\begin{cases} \alpha,&\;\;X_{\tau_n,-}\leq a,\\ \beta,&\;\;X_{\tau_n,-}\geq b.\end{cases}
\end{align*}
There is no hope to find the parameters $(a,\alpha, \beta,b)$ and the associated value function explicitly in great generality. Even ordinary optimal stopping problems for L\'evy processes are very hard to solve. A new method developed over the last years is to make the Ansatz to write the value function as an expectation of the running maximum or minimum of the process evaluated at an independent $Exp(r)$-distributed time $T$, see e.g. \cite{NS2}, \cite{M}, \cite{Su}, \cite{DLU}, \cite{CST}. We will show in the following, that one can be hopeful to use the same approach also for impulse control problems. \\
The main tool is the following representation of general non-negative $r$-excessive functions $k$ (under some conditions):
\[k(x)=\int_\R G_r(x,y)\sigma(dy)\]
for some Radon measure $\sigma=\sigma_k$, where $G_r(x,y)$ denotes the Green kernel of the process; this representation is based on the Riesz representation theorem, see \cite{MS} and the references therein for a more detailed discussion. The measure $\sigma$ does not charge the points in $\R$, where $k$ is $r$-harmonic. Now, we use this representation for $k=v-\ov f$. By equation \eqref{eq:harm} we see that $v-\ov f$ is typically $r$-harmonic on $(a,b)$. Therefore, $\sigma$ has support on $(a,b)^c$. We may write
\[(v-\ov f)(x)=\int_{(-\infty,a]} G_r(x,y)\sigma(dy)+\int_{[b,\infty)} G_r(x,y)\sigma(dy)\]
for all $x\in E$. Writing $M=\sup\{X_t:t<T\}$ and $I=\inf\{X_t:t<T\}$ and assume that $M$ and $I$ have densities $f_M,f_I$, by the Wiener-Hopf-factorization, we have the representation
\[r G_r(x,y)=\begin{cases}
\int_{-\infty}^{y-x}f_I(t)f_M(y-x-t)dt&,\;y-x<0,\\
\int_{y-x}^{\infty}f_M(t)f_I(y-x-t)dt&,\;y-x>0,
\end{cases}
\]
see \cite{MS}. We obtain for all $x\in(a,b)$
\begin{align*}
(v-\ov f)(x)=&\int_{(-\infty,a]} G_r(x,y)\sigma(dy)+\int_{[b,\infty)} G_r(x,y)\sigma(dy)\\
=&r^{-1}\int_{-\infty}^a \int_{-\infty}^{y-x}f_I(t)f_M(y-x-t)dt\sigma(dy)\\&+r^{-1}\int_b^\infty \int_{y-x}^{\infty}f_M(t)f_I(y-x-t)dt\sigma(dy)\\
=&r^{-1}\int_{-\infty}^{a-x} f_I(t)\int_{x+t}^{a}f_M(y-x-t)\sigma(dy)dt\\&+r^{-1}\int_{b-x}^\infty f_M(t)\int_{b}^{x+t}f_I(y-x-t)\sigma(dy)dt\\
=&\E_x(Q_*(I);I\leq a)+\E_x(Q^*(M);M\geq b),
\end{align*}
where 
\[Q_*(z)=r^{-1}\int_{z}^{a}f_M(y-z)\sigma(dy),\;\;\; Q^*(z)=r^{-1}\int_{b}^{z}f_I(y-z)\sigma(dy).\]
This gives a representation of the value function in terms of the running maximum and minimum, as desired. On the other hand, functions of the form 
\[\int_{(-\infty,a]} G_r(x,y)\sigma(dy)+\int_{[b,\infty)} G_r(x,y)\sigma(dy)\]
are $r$-superharmonic and one can start with these functions to find a candidate solution for the value function. This discussion opens the door to use the strong methods developed for optimal stopping with underlying L\'evy processes for impulse control problems. Since carrying out the details for a concrete example is quite lengthy, we stop the discussion here, but treat an interesting example from portfolio optimization for fixed transaction costs for general L\'evy processes in a forthcoming article. 
\section*{Acknowledgement}
I would like to thank Ralf Korn and J\"orn Sass for a discussion about the applicability of the theory of optimal stopping for L\'evy processes to impulse control (see Subsection \ref{subsec:levy}), that was the starting point for this paper. Furthermore, I would like to thank A. Irle and A. Ludwig for useful remarks. 

\bibliographystyle{plain}
\bibliography{literatur}

  \begin{appendix}
  \section{On the formal definition of impulse control strategies}\label{appendix:impulse}
Now, we give a more formal definition of an impulse control strategy and the corresponding controlled process by following the classical construction, see e.g. \cite{S} to give a reference in\ English:\\
We consider the new probability space $\tilde\Omega=\Omega^\N$. Then, 
\begin{align*}
&\tau_1\mbox{ is an }\left(\mathcal{F}_t\otimes\bigotimes_{k\geq2}\{\emptyset,\Omega\}\right)_{t\geq0}\mbox{ stopping time},\\
&\gamma_1\mbox{ is }\mathcal{F}_{\tau_1}\otimes\bigotimes_{k\geq2}\{\emptyset,\Omega\}\mbox{-mesurable}.
\end{align*}

and, more generally, for all $n\in\N$
\begin{align*}
&\tau_n\mbox{ is an }\left(\bigotimes_{l=1}^n\mathcal{F}_t\otimes\bigotimes_{k\geq n+1}\{\emptyset,\Omega\}\right)_{t\geq0}\mbox{ stopping time},\\
&\gamma_n\mbox{ is }\bigotimes_{l=1}^n\mathcal{F}_{\tau_n}\otimes\bigotimes_{k\geq n+1}\{\emptyset,\Omega\}=:\tilde{\mathcal{F}}^{(n)}_{\tau_n}\mbox{-mesurable}.
\end{align*}
At each random time $\tau_n$ an impulse is exercised and the process is restarted at the new state $\gamma_n$. $\tau_n$ and $\gamma_n$ only depend on the first $n$ coordinates in $\tilde\Omega$. The restarted processes is described by the $(n+1)$-th coordinate of $\tilde\Omega$. More precisely, there exists a family  $\Pro^S_x,\;x\in E,$ of probability measures on $\tilde\Omega$, that is characterized by the distributions of the coordinate processes as follows: For all $s\geq0,n\in\N,A_1,...,A_{n+1}$ measurable,  
\begin{align*}
&\Pro^S_x(X^{(1)}_{\tau_n+s}\in A_1,...,X^{(n)}_{\tau_n+s}\in A_n,,X^{(n+1)}_{\tau_n+s}\in A_{n+1}|\tilde{\mathcal{F}}^{(n)}_{\tau_n})\\
=&\delta_{X^{(1)}_{\tau_1}}(A_1)\cdot...\cdot \delta_{X^{(n)}_{\tau_n}}(A_n)\Pro_{\gamma_n}(X_s\in A_{n+1})\quad \mbox{on }\{\tau_n+s<\tau_{n+1}\},
\end{align*}
where $\delta$ denotes the Dirac measure. The trajectories of the controlled process are then given by the trajectories of the copies $X^n,n=1,2,...$ of $X$ on $\tilde\Omega$ as follows:
\[X_t(\omega)=X_t^{n}(\omega_n)\mbox{ for }t\in[\tau_{n-1},\tau_n),\;X_{\tau_n}(\omega)=\gamma_n(\omega_1,...,\omega_n),\;\tau_0=0.\]
  \end{appendix}

\end{document}